\documentclass[oneside,british]{amsart}
\usepackage[T1]{fontenc}
\usepackage[latin9]{inputenc}
\usepackage{amsthm}
\usepackage{amssymb}
\usepackage{setspace}
\usepackage{esint}
\makeatletter
\numberwithin{equation}{section}
\numberwithin{figure}{section}
\theoremstyle{plain}
\newtheorem{thm}{\protect\theoremname}[section]
  \theoremstyle{definition}
  \newtheorem{defn}[thm]{\protect\definitionname}
  \theoremstyle{remark}
  \newtheorem{rem}[thm]{\protect\remarkname}
  \theoremstyle{remark}
  \newtheorem{claim}[thm]{\protect\claimname}
  \theoremstyle{plain}
  \newtheorem{lem}[thm]{\protect\lemmaname}
  \theoremstyle{plain}
  \newtheorem{cor}[thm]{\protect\corollaryname}
  \theoremstyle{plain}
  \newtheorem{prop}[thm]{\protect\propositionname}

\makeatother

\usepackage{babel}
  \providecommand{\claimname}{Claim}
  \providecommand{\corollaryname}{Corollary}
  \providecommand{\definitionname}{Definition}
  \providecommand{\lemmaname}{Lemma}
  \providecommand{\propositionname}{Proposition}
  \providecommand{\remarkname}{Remark}
\providecommand{\theoremname}{Theorem}

\begin{document}

\title{A dimension gap for continued fractions with independent digits -
the non stationary case}

\author{Ariel Rapaport}

\date{March 9, 2017}



\thanks{Supported by ERC grant 306494}
\begin{abstract}
We show there exists a constant $0<c_{0}<1$ such that the dimension
of every measure on $[0,1]$, which makes the digits in the continued
fraction expansion independent, is at most $1-c_{0}$. This extends
a result of Kifer, Peres and Weiss from 2001, which established this
under the additional assumption of stationarity. For $k\ge1$ we prove
an analogues statement for measures under which the digits form a
$*$-mixing $k$-step Markov chain. This is also generalized to the
case of $f$-expansions. In addition, we construct for each $k$ a
measure, which makes the continued fraction digits a stationary and
$*$-mixing $k$-step Markov chain, with dimension at least $1-2^{3-k}$.
\end{abstract}

\maketitle

\section{Introduction}

Let $X$ denote the set of irrational numbers in $(0,1)$. It is well
known each $x\in X$ has a unique continued fraction expansion of
the form
\[
x=\frac{1}{A_{1}(x)+\frac{1}{A_{2}(x)+\frac{1}{A_{3}(x)+\cdots}}}\:,
\]
where $A_{1}(x),A_{2}(x),...$ are positive integers. Given a probability
measure $\nu$ on $X$, each $A_{n}$ defines a random variable on
$(X,\nu)$ and the digits $\{A_{n}\}_{n=1}^{\infty}$ form a discrete
time stochastic process.

In 1966, Chatterji \cite{C} has shown every probability measure $\nu$
on $[0,1]$, which makes the digits in the continued fraction expansion
independent variables, is singular with respect to the Lebesgue measure.
In 2001, Kifer, Peres and Weiss \cite{KPW} have proven that $\dim_{H}\nu\le1-c$,
if in addition the digits are identically distributed. Here $0<c<1$
is a global constant, independent of $\nu$, and $\dim_{H}\nu$ denotes
the Hausdorff dimension of $\nu$, which is defined in Section \ref{S2}
below. In this paper we show the result from \cite{KPW} remains true,
even if the digits are independent but not necessarily identically
distributed.

Assuming $A_{1},A_{2},...$ are i.i.d. with $\mathbb{E}[\log A_{1}]<\infty$
and $H(A_{1})<\infty$, where $H(A_{1})$ is the entropy of $A_{1}$,
Kinney and Pitcher \cite{KP} have proven that 
\begin{equation}
\dim_{H}\nu=\frac{H(A_{1})}{-\int_{0}^{1}\log x^{2}\:d\nu}\:.\label{E100}
\end{equation}
The Gauss measure
\[
\mu_{G}(E)=\frac{1}{\log2}\int_{E}\frac{dx}{1+x}
\]
is the unique equilibrium state of the Gauss map $Tx=\frac{1}{x}\:\left(\mathrm{mod}\:1\right)$
with respect to the function $x\rightarrow\log x^{2}$. This follows
from the thermodynamic formalism approach of Walters \cite{W}. Hence
under the i.i.d. assumption
\[
0=\int_{0}^{1}\log x^{2}\:d\mu_{G}(x)+h_{\mu_{G}}(T)>\int_{0}^{1}\log x^{2}\:d\nu(x)+h_{\nu}(T),
\]
where $h_{\eta}(T)$ is the entropy of $T$ with respect to a $T$-invariant
measure $\eta$. Since $h_{\nu}(T)=H(A_{1})$, we get from (\ref{E100})
that $\dim_{H}\nu<1$ in this case. When $A_{1},A_{2},...$ are not
identically distributed the formula (\ref{E100}) is no longer valid,
and so it is not even clear that $\dim_{H}\nu$ is strictly less than
$1$. As mentioned above, we shall show that there exists a global
constant $c_{0}>0$ such that $\dim_{H}\nu\le1-c_{0}$, assuming $A_{1},A_{2},...$
are independent.

We actually prove more generally that for every integer $k\ge0$ there
exists $0<c_{k}<1$, which depends only on $k$, such that $\dim_{H}\nu\le1-c_{k}$
if the digits form a $k$-step Markov chain which is $*$-mixing.
This is the main result of this paper. The $*$-mixing condition was
introduced in \cite{BHK}, and is a bit less restrictive than the
more familiar $\psi$-mixing condition. The definitions are given
in Section \ref{S2}. In the last section we generalize our main result
to the case of $f$-expansions.

Given $k\ge0$ it was shown in \cite{KPW} that there exists $0<c_{k}'<1$,
for which $\dim_{H}\nu\le1-c_{k}'$ whenever $\nu$ makes the digits
a stationary and ergodic $k$-step Markov chain. Our proof is a modification
of the argument given there for this result. We shall also construct
for each $k$ a measure $\nu_{k}$, under which the digits form a
stationary and $\psi$-mixing $k$-step Markov chain, with $\dim_{H}\nu_{k}\ge1-2^{3-k}$.
This of course shows $c_{k}\mbox{ and }c_{k}'$ are at most $2^{3-k}$.

The rest of the paper is organized as follows. In Section \ref{S2}
we give some necessary definitions and state our results. In Section
\ref{S3} we establish a uniform bound on the dimension of subsets
of $X$, which are defined via certain digit frequencies. This is
the key ingredient in the proof of our main result, which is carried
out in Section \ref{S4}. In Section \ref{S5} we construct the measures
$\nu_{k}$ mentioned above. In Section \ref{S6} we generalize our
main result to the setup of $f$-expansions.$\newline$

\textbf{Acknowledgement. }This paper is a part of the author's PhD
thesis conducted at the Hebrew University of Jerusalem. I would like
to thank my advisor Professor Yuri Kifer, for suggesting to me the
problem studied in this paper, and for many helpful discussions.

\section{\label{S2}Preliminaries and results}

First, we define the mixing conditions mentioned above. Given random
variables $\{A_{i}\}_{i\in I}$, all defined on the same probability
space, denote by $\sigma\{A_{i}\}_{i\in I}$ the smallest $\sigma$-algebra
with respect to which each $A_{i}$ is measurable.
\begin{defn}
A sequence of random variables $\{A_{n}\}_{n=1}^{\infty}$ is called
$*$-mixing if there exist an integer $N\ge1$ and a real valued function
$f$, defined on the integers $n\ge N$, such that
\begin{itemize}
\item $f$ is non-increasing with $\underset{n}{\lim}\:f(n)=0$, and
\item if $n\ge N$, $m\ge1$, $E\in\sigma\{A_{1},...,A_{m}\}$ and $F\in\sigma\{A_{m+n},A_{m+n+1},...\}$
then
\[
\left|\mathbb{P}(E\cap F)-\mathbb{P}(E)\mathbb{P}(F)\right|\le f(n)\mathbb{P}(E)\mathbb{P}(F)\:.
\]

\end{itemize}
If such an $f$ exists for $N=1$ the sequence is said to be $\psi$-mixing.
\end{defn}

\begin{rem}
\label{R2.2}A sequence of independent random variables is clearly
$\psi$-mixing. It is not hard to show that the $\psi$-mixing condition
is satisfied for a finite state Markov chain $\{A_{n}\}_{n=1}^{\infty}$,
with state space $S$, for which
\[
\inf\{\mathbb{P}(A_{n+1}=j\mid A_{n}=i)\::\:n\ge1\mbox{ and }i,j\in S\}>0\:.
\]
Examples of $*$-mixing countable state Markov chains can be found
in Section 3 of \cite{BHK}. Another important example of a $\psi$-mixing
sequence is obtained by the continued fraction digits with respect
to the Gauss measure $\mu_{G}$ (see \cite{A} or \cite{H}).
\end{rem}

Set $X=(0,1)\setminus\mathbb{Q}$ and for each $x\in X$ and $i\ge1$
let $\alpha_{i}(x)\in\mathbb{N}:=\{1,2,...\}$ be the $i$'th digit
in the continued fraction expansion of $x$, i.e.
\[
x=\frac{1}{\alpha_{1}(x)+\frac{1}{\alpha_{2}(x)+\frac{1}{\alpha_{3}(x)+\cdots}}}\:.
\]
Given $a_{1},a_{2},...\in\mathbb{N}$ denote by $[a_{1},a_{2},...]$
the unique $x\in X$ with $\alpha_{i}(x)=a_{i}$ for $i\ge1$. For
$E\subset X$ write $\dim_{H}(E)$ for the Hausdorff dimension of
$E$. Given a Borel probability measure $\nu$ on $X$ its Hausdorff
dimension is defined by
\[
\dim_{H}(\nu)=\inf\{\dim_{H}(E)\::\:E\subset X\mbox{ is a Borel set with }\nu(E)=1\}\:.
\]
The following theorem is our main result.
\begin{thm}
\label{T3}Let $\{A_{n}\}_{n=1}^{\infty}$ be $\mathbb{N}$-valued
random variables and let $k\ge0$. Assume $\{A_{n}\}_{n=1}^{\infty}$
is a $k$-step Markov chain (when $k=0$ this means $A_{1},A_{2},...$
are independent) which is $*$-mixing. Let $\nu$ be the distribution
of the random variable $[A_{1},A_{2},...]$. Then $\dim_{H}(\nu)\le1-c_{k}$,
where $0<c_{k}<1$ is a constant depending only on $k$.
\end{thm}

\begin{rem}
As mentioned in the introduction, it was shown in \cite{KPW} that
there exists $0<c_{k}'<1$, for which $\dim_{H}\nu\le1-c_{k}'$ whenever
$\nu$ makes the continued fraction digits a stationary and ergodic
$k$-step Markov chain.
\end{rem}

It might be desirable to estimate $c_{k}$ and $c_{k}'$. The next
claim shows these constants are at most $2^{3-k}$.
\begin{claim}
\label{C5}For each $k\ge3$ there exits an $\mathbb{N}$-valued $k$-step
stationary and $\psi$-mixing Markov chain $\{A_{n}\}_{n=1}^{\infty}$
with $\dim_{H}(\nu)\ge1-2^{3-k}$, where $\nu$ is the distribution
of $[A_{1},A_{2},...]$.
\end{claim}

The main ingredient in the proof of Theorem \ref{T3} is Theorem \ref{T6}
stated below, for which we need some more notations. Let $T:X\rightarrow X$
be the Gauss map, which is defined by 
\[
Tx=\frac{1}{x}\:\left(\mathrm{mod}\:1\right)\mbox{ for }x\in X\:.
\]
Denote by $\mu_{G}$ the Gauss measure, which satisfies
\[
\mu_{G}(E)=\frac{1}{\log2}\int_{E}\frac{dx}{1+x}\mbox{ for every Borel set }E\subset X\:.
\]
It is well known that $\mu_{G}$ is invariant and ergodic with respect
to $T$. For $(a_{1},...,a_{k})=\mathbf{a}\in\mathbb{N}^{k}$ set
\[
I_{\mathbf{a}}=\{x\in X\::\:\alpha_{i}(x)=a_{i}\mbox{ for each }1\le i\le k\},
\]
and define $\mathbb{I}_{\mathbf{a}}:X\rightarrow\{0,1\}$ by
\[
\mathbb{I}_{\mathbf{a}}(x)=\begin{cases}
1 & \mbox{, if }x\in I_{\mathbf{a}}\\
0 & \mbox{, if }x\notin I_{\mathbf{a}}
\end{cases}\mbox{ for }x\in X\:.
\]
Given $L>1$ denote by $\mathcal{Q}_{L}$ the set of maps $q:\mathbb{N}\rightarrow\mathbb{N}$
with
\[
q(n+1)>q(n)\mbox{ for each }n\in\mathbb{N}
\]
and
\[
\underset{n\rightarrow\infty}{\liminf}\:\frac{q(n)}{n}<L\:.
\]
For $q\in\mathcal{Q}_{L}$, $\mathbf{a}\in\cup_{k=1}^{\infty}\mathbb{N}^{k}$
and $\delta>0$ define
\[
\Gamma_{q,\mathbf{a}}^{\delta}=\{x\in X\::\:\underset{n\rightarrow\infty}{\liminf}\:\left|\frac{1}{n}\sum_{i=1}^{n}\mathbb{I}_{\mathbf{a}}(T^{q(i)}x)-\mu_{G}(I_{\mathbf{a}})\right|>\delta\}\:.
\]

\begin{thm}
\label{T6}For every $L>1$ and $\delta>0$ there exists $0<c_{L,\delta}<1$
with
\[
\sup\{\dim_{H}(\Gamma_{q,\mathbf{a}}^{\delta})\::\:q\in\mathcal{Q}_{L},\:\mathbf{a}\in\cup_{k=1}^{\infty}\mathbb{N}^{k}\}\le1-c_{L,\delta}\:.
\]

\end{thm}

\begin{rem}
The proof of theorem \ref{T6} resembles the proof of the main result
(Theorem 2.1) of \cite{KPW}. There an upper bound, which depends
only on $\delta$, is obtained for the dimension of sets of the form
\begin{equation}
\{x\in X\::\:\underset{n\rightarrow\infty}{\limsup}\:\left|\frac{1}{n}\sum_{i=1}^{n}\mathbb{I}_{\mathbf{a}}(T^{i}x)-\mu_{G}(I_{\mathbf{a}})\right|>\delta\}\}\:.\label{E101}
\end{equation}
Here we need to consider the families $\mathcal{Q}_{L}$, and the
more general averages
\[
\frac{1}{n}\sum_{i=1}^{n}\mathbb{I}_{\mathbf{a}}(T^{q(i)}x),
\]
due to the lack of stationarity. As a result we must define $\Gamma_{q,\mathbf{a}}^{\delta}$
with $\liminf$, as opposed to the sets (\ref{E101}) which are defined
with $\limsup$.
\end{rem}

\section{\label{S3}Proof of Theorem \ref{T6}}

The following large deviations estimate will be needed. Its proof
is almost identical to the proof of Lemma 3.1 from \cite{KPW}, but
we include it here for completeness.
\begin{lem}
\label{L3.1}Suppose $\mathbf{S}=\{\eta_{n}\}_{n=1}^{\infty}$ is
a stationary and $*$-mixing sequence of random variables. Let $k\ge1$
and $F:\mathbb{R}^{k}\rightarrow\{0,1\}$, set
\[
p=\mathbb{P}\{F(\eta_{1},...,\eta_{k})=1\},
\]
and let $q:\mathbb{N}\rightarrow\mathbb{N}$ be strictly increasing.
Then for every $\delta>0$ there exists a constant $M=M(\mathbf{S},\delta,k)>1$,
independent of $q$ and $F$, such that for every $n\ge1$,
\[
\mathbb{P}\left\{ \left|\frac{1}{n}\sum_{i=1}^{n}F(\eta_{q(i)},...,\eta_{q(i)+k-1})-p\right|>\delta\right\} \le M\cdot e^{-n/M}\:.
\]

\end{lem}

\begin{proof}
Fix $\delta>0$, then since $\mathbf{S}$ is $*$-mixing there exists
$M\in\mathbb{N}$ with 
\begin{equation}
\left|\mathbb{P}(E_{1}\cap E_{2})-\mathbb{P}(E_{1})\mathbb{P}(E_{2})\right|\le\frac{\delta^{2}}{2}\mathbb{P}(E_{1})\mathbb{P}(E_{2})\label{E96}
\end{equation}
for each $m\ge1$, $E_{1}\in\sigma\{\eta_{1},...,\eta_{m+k-1}\}$
and $E_{2}\in\sigma\{\eta_{m+M},\eta_{m+M+1},...\}$. For $i\ge1$
set $\xi_{i}=F(\eta_{i},...,\eta_{i+k-1})$, fix $n\ge M$, and write
\[
A_{n}=\left\{ \left|\frac{1}{n}\sum_{i=1}^{n}\xi_{q(i)}-p\right|>\delta\right\} \:.
\]
Let $N$ be the integral part of $n/M$, and for $1\le j\le M$ set
\[
B_{n,j}=\left\{ \left|\frac{1}{N}\sum_{i=0}^{N-1}\xi_{q(j+iM)}-p\right|>\delta-\frac{1}{N}\right\} \:.
\]
Clearly $A_{n}\subset\cup_{j=1}^{M}B_{n,j}$, hence
\begin{equation}
\mathbb{P}(A_{n})\le\sum_{j=1}^{M}\mathbb{P}(B_{n,j})\:.\label{E94}
\end{equation}

Fix $1\le j\le M$, and for $\epsilon_{0},...,\epsilon_{N-1}\in\{0,1\}$
write
\[
\mathcal{C}_{\epsilon_{0},...,\epsilon_{N-1}}=\{\xi_{q(j+iM)}=\epsilon_{i}\mbox{ for each }0\le i<N\}\:.
\]
Let $\zeta_{0},\zeta_{1},...$ be independent $\{0,1\}$-valued random
variables with mean $p$. Since $q$ is strictly increasing it follows
easily from (\ref{E96}) that,
\begin{align*}
\mathbb{P}(\mathcal{C}_{\epsilon_{0},...,\epsilon_{N-1}}) & \le(1+\frac{\delta^{2}}{2})^{N}\cdot\prod_{i=0}^{N-1}\mathbb{P}\{\xi_{q(j+iM)}=\epsilon_{i}\}\\
 & \le e^{\delta^{2}N/2}\cdot\mathbb{P}\{\zeta_{i}=\epsilon_{i}\mbox{ for each }0\le i<N\}\:.
\end{align*}
Set $Z=\sum_{i=0}^{N-1}\zeta_{i}$, then $Z$ is a binomial random
variable with parameters $N$ and $p$, and
\begin{align}
\mathbb{P}(B_{n,j}) & =\sum_{\left|\sum_{i=0}^{N-1}\epsilon_{i}-Np\right|>N\delta-1}\mathbb{P}(\mathcal{C}_{\epsilon_{0},...,\epsilon_{N-1}})\nonumber \\
 & \le e^{\delta^{2}N/2}\cdot\mathbb{P}\{\left|Z-Np\right|>N\delta-1\}\:.\label{E95}
\end{align}
By the exponential estimate for the binomial distribution (see e.g.
Cor. A.1.7 in \cite{AS}) we have for $N\ge4/\delta$,
\[
\mathbb{P}\{\left|Z-Np\right|>N\delta-1\}\le2e^{-N\delta^{2}}\:.
\]
This together with (\ref{E95}) gives,
\[
\mathbb{P}(B_{n,j})\le2e^{-\delta^{2}N/2}\mbox{ for each }1\le j\le M\:.
\]
The lemma now follows from (\ref{E94}).
\end{proof}

As mentioned in Remark \ref{R2.2}, the sequence $\{\alpha_{i}\}_{i=1}^{\infty}$
is $\psi$-mixing with respect to $\mu_{G}$. From this and Lemma
\ref{L3.1} we get the following corollary.
\begin{cor}
\label{P8}Given $k\ge1$ and $\delta>0$ there exists a constant
$M=M(\delta,k)>1$, such that for every strictly increasing $q:\mathbb{N}\rightarrow\mathbb{N}$,
$\mathbf{a}\in\mathbb{N}^{k}$ and $n\ge1$,
\[
\mu_{G}\left\{ x\in X\::\:\left|\frac{1}{n}\sum_{i=1}^{n}\mathbb{I}_{\mathbf{a}}(T^{q(i)}x)-\mu_{G}(I_{\mathbf{a}})\right|>\delta\right\} \le M\cdot e^{-n/M}\:.
\]

\end{cor}

Given $x\in X$ and $n\ge1$ write $J_{n}(x)=I_{(\alpha_{1}(x),...,\alpha_{n}(x))}$.
Let $\mathcal{L}$ be the Lebesgue measure, and write $|I|=\mathcal{L}(I)$
for $I\subset X$ . For $s\ge0$ let $\mathcal{H}^{s}$ be the $s$-dimensional
Hausdorff measure on $X$. For $\eta>0$ and $E\subset X$ write
\[
\mathcal{H}_{\eta}^{s}(E)=\inf\{\sum_{i=1}^{\infty}|I_{i}|^{s}\::\:E\subset\cup_{i=1}^{\infty}I_{i}\mbox{ and }|I_{i}|\le\eta\},
\]
then
\[
\underset{\eta\downarrow0}{\lim}\:\mathcal{H}_{\eta}^{s}(E)=\mathcal{H}^{s}(E)\:.
\]
Given $n\ge1$ write
\[
\beta_{n}=\sup\{|I_{\mathbf{a}}|\::\:\mathbf{a}\in\mathbb{N}^{n}\},
\]
then $\beta_{n}\overset{n}{\rightarrow}0$.
\begin{proof}[Proof of Theorem \ref{T6}]
Let $\delta>0$, $L>1$, $q\in\mathcal{Q}_{L}$, $k\ge1$ and $\mathbf{a}\in\mathbb{N}^{k}$.
Given $\lambda>0$ set,
\[
\mathcal{E}_{\lambda}:=\cap_{j=1}^{\infty}\cup_{n=j}^{\infty}\{x\in X\::\:|J_{n}(x)|\le e^{-\lambda n}\}\:.
\]
By Theorem 4.1 in \cite{KPW} there exists $\lambda>0$ with $\dim_{H}\mathcal{E}_{\lambda}<1$.
For $N\ge1$ set
\[
\Gamma_{q,\mathbf{a}}^{\delta,N}:=\left\{ x\in X\::\:\begin{array}{c}
\left|\frac{1}{n}\sum_{i=1}^{n}\mathbb{I}_{\mathbf{a}}(T^{q(i)}x)-\mu_{G}(I_{\mathbf{a}})\right|>\delta,\\
|J_{q(n)+k}(x)|\ge e^{-\lambda(q(n)+k)}
\end{array}\mbox{ for all }n\ge N\right\} ,
\]
then
\begin{equation}
\Gamma_{q,\mathbf{a}}^{\delta}\setminus\mathcal{E}_{\lambda}\subset\cup_{N=1}^{\infty}\Gamma_{q,\mathbf{a}}^{\delta,N}\:.\label{E1}
\end{equation}
Fix $N\ge1$ and for $n\ge1$ set
\[
\Upsilon_{q,\mathbf{a}}^{\delta,n}:=\left\{ x\in X\::\:\begin{array}{c}
\left|\frac{1}{n}\sum_{i=1}^{n}\mathbb{I}_{\mathbf{a}}(T^{q(i)}x)-\mu_{G}(I_{\mathbf{a}})\right|>\delta,\\
|J_{q(n)+k}(x)|\ge e^{-\lambda(q(n)+k)}
\end{array}\right\} ,
\]
then $\Gamma_{q,\mathbf{a}}^{\delta,N}\subset\Upsilon_{q,\mathbf{a}}^{\delta,n}$
for all $n\ge N$.

Let $M=M(\delta,k)>1$ be as in Corollary \ref{P8}, set $s=1-\frac{1}{\lambda LM}$
and let $\eta>0$. From $q\in\mathcal{Q}_{L}$ we get $\underset{n\rightarrow\infty}{\liminf}\:\frac{q(n)}{n}<L$.
From this and $\beta_{n}\overset{n}{\rightarrow}0$ it follows that
there exists $n\ge N$ such that $\beta_{n}<\eta$ and $q(n)<nL$.
By the definition of $\Upsilon_{q,\mathbf{a}}^{\delta,n}$ there exists
$B_{n}\subset\mathbb{N}^{q(n)+k}$ with $\Upsilon_{q,\mathbf{a}}^{\delta,n}=\cup_{\mathbf{b}\in B_{n}}I_{\mathbf{b}}$.
From Corollary \ref{P8} we get
\[
\mu_{G}(\Upsilon_{q,\mathbf{a}}^{\delta,n})\le M\cdot e^{-n/M}\:.
\]
Since
\[
r:=\underset{x\in[0,1]}{\min}\:\frac{d\mu_{G}}{d\mathcal{L}}(x)>0,
\]
it follows
\begin{equation}
\sum_{\mathbf{b}\in B_{n}}|I_{\mathbf{b}}|=\mathcal{L}(\Upsilon_{q,\mathbf{a}}^{\delta,n})\le r^{-1}\cdot\mu_{G}(\Upsilon_{q,\mathbf{a}}^{\delta,n})\le r^{-1}M\cdot e^{-n/M}\:.\label{E4}
\end{equation}
From
\[
\Gamma_{q,\mathbf{a}}^{\delta,N}\subset\Upsilon_{q,\mathbf{a}}^{\delta,n}=\cup_{\mathbf{b}\in B_{n}}I_{\mathbf{b}}
\]
and since $|I_{\mathbf{b}}|\le\beta_{n}<\eta$ for every $\mathbf{b}\in B_{n}$,
\begin{equation}
\mathcal{H}_{\eta}^{s}(\Gamma_{q,\mathbf{a}}^{\delta,N})\le\sum_{\mathbf{b}\in B_{n}}|I_{\mathbf{b}}|^{s}\le(\underset{\mathbf{b}\in B_{n}}{\inf}\:|I_{\mathbf{b}}|)^{s-1}\cdot\sum_{\mathbf{b}\in B_{n}}|I_{\mathbf{b}}|\:.\label{E9}
\end{equation}
By the definition of $\Upsilon_{q,\mathbf{a}}^{\delta,n}$,
\[
|I_{\mathbf{b}}|\ge e^{-\lambda(q(n)+k)}\mbox{ for every }\mathbf{b}\in B_{n}\:.
\]
Hence from (\ref{E9}), (\ref{E4}), $q(n)<nL$ and $s=1-\frac{1}{\lambda LM}$,
\begin{multline*}
\mathcal{H}_{\eta}^{s}(\Gamma_{q,\mathbf{a}}^{\delta,N})\le e^{\lambda(q(n)+k)(1-s)}\cdot r^{-1}M\cdot e^{-n/M}\\
\le r^{-1}Me^{\lambda k}\cdot\exp(n(\lambda L(1-s)-M^{-1}))=r^{-1}Me^{\lambda k}\:.
\end{multline*}
As this holds for every $\eta>0$
\[
\mathcal{H}^{s}(\Gamma_{q,\mathbf{a}}^{\delta,N})=\underset{\eta\downarrow0}{\lim}\:\mathcal{H}_{\eta}^{s}(\Gamma_{q,\mathbf{a}}^{\delta,N})\le r^{-1}Me^{\lambda k}<\infty,
\]
and so
\[
\dim_{H}(\Gamma_{q,\mathbf{a}}^{\delta,N})\le s=1-\frac{1}{\lambda LM}\:.
\]
As this holds for every $N\ge1$ it follows from (\ref{E1}) that,
\begin{equation}
\dim_{H}(\Gamma_{q,\mathbf{a}}^{\delta}\setminus\mathcal{E}_{\lambda})\le\underset{N\ge1}{\sup}\:\dim_{H}(\Gamma_{q,\mathbf{a}}^{\delta,N})\le1-\frac{1}{\lambda L\cdot M(\delta,k)}\:.\label{E3}
\end{equation}

We shall now complete the proof of the theorem. We continue to fix
$\delta>0$ and $L>1$. Let
\[
k_{\delta}=\inf\{k\ge1\::\:\underset{\mathbf{a}\in\mathbb{N}^{k}}{\sup}\:\mu_{G}(I_{\mathbf{a}})<\frac{\delta}{2}\},
\]
then clearly $k_{\delta}<\infty$. For $q\in\mathcal{Q}_{L}$, $k\ge k_{\delta}$
and $(a_{1},...,a_{k})=\mathbf{a}\in\mathbb{N}^{k}$,
\begin{equation}
\Gamma_{q,\mathbf{a}}^{\delta/2}\supset\left\{ x\in X\::\:\underset{n\rightarrow\infty}{\liminf}\:\frac{1}{n}\sum_{i=1}^{n}\mathbb{I}_{\mathbf{a}}(T^{q(i)}x)>\delta\right\} \supset\Gamma_{q,\mathbf{a}}^{\delta}\:.\label{E2}
\end{equation}
Set $\mathbf{a}_{\delta}=(a_{1},...,a_{k_{\delta}})$, then since
$\mathbb{I}_{\mathbf{a}_{\delta}}\ge\mathbb{I}_{\mathbf{a}}$ it follows
from (\ref{E2}) that $\Gamma_{q,\mathbf{a}_{\delta}}^{\delta/2}\supset\Gamma_{q,\mathbf{a}}^{\delta}$,
and so
\[
\dim_{H}(\Gamma_{q,\mathbf{a}_{\delta}}^{\delta/2})\ge\dim_{H}(\Gamma_{q,\mathbf{a}}^{\delta})\:.
\]
This together with (\ref{E3}) gives
\begin{multline*}
\sup\{\dim_{H}(\Gamma_{q,\mathbf{a}}^{\delta})\::\:q\in\mathcal{Q}_{L},\:\mathbf{a}\in\cup_{k=1}^{\infty}\mathbb{N}^{k}\}\\
\le\max\{\dim_{H}(\mathcal{E}_{\lambda}),\underset{1\le k\le k_{\delta}}{\max}(1-\frac{1}{\lambda L\cdot M(\delta/2,k)})\}<1,
\end{multline*}
which completes the proof of the theorem.
\end{proof}

\section{\label{S4}Proof of the main result}
\begin{proof}[Proof of Theorem \ref{T3}]
Fix $k\ge0$, let $\{A_{n}\}_{n=1}^{\infty}$ an $\mathbb{N}$-valued
$k$-step Markov chain which is $*$-mixing, and let $\nu$ be the
distribution of $[A_{1},A_{2},...]$. Given words $\mathbf{a}\in\mathbb{N}^{m}$
and $\mathbf{b}\in\mathbb{N}^{l}$ we denote by $\mathbf{ab}\in\mathbb{N}^{m+l}$
their concatenation. As noted in observation 2.2 in \cite{KPW}, the
continued fraction digits under $\mu_{G}$ do not form a $k$-step
Markov chain. It follows that there exist $m\in\mathbb{N}$, $\mathbf{a}\in\mathbb{N}^{k}$,
$\mathbf{b}\in\mathbb{N}^{m}$ and $c\in\mathbb{N}$ with
\[
\frac{\mu_{G}(I_{\mathbf{ba}c})}{\mu_{G}(I_{\mathbf{ba}})}\ne\frac{\mu_{G}(I_{\mathbf{a}c})}{\mu_{G}(I_{\mathbf{a}})},
\]
and so
\begin{equation}
\delta:=\left|\mu_{G}(I_{\mathbf{ba}c})-\frac{\mu_{G}(I_{\mathbf{a}c})\cdot\mu_{G}(I_{\mathbf{ba}})}{\mu_{G}(I_{\mathbf{a}})}\right|>0\:.\label{E8}
\end{equation}
If $k=0$, i.e. when $A_{1},A_{2},...$ are independent, $\mathbf{a}$
is the empty word and $I_{\mathbf{a}}=X$. Let $\mu_{G}(I_{\mathbf{a}})>\epsilon>0$
be such that if $p_{1},p_{2},p_{3}\in[0,1]$ satisfy
\[
|p_{1}-\mu_{G}(I_{\mathbf{a}c})|,\:|p_{2}-\mu_{G}(I_{\mathbf{ba}})|,\:|p_{3}-\mu_{G}(I_{\mathbf{a}})|\le\epsilon,
\]
then
\begin{equation}
|\frac{p_{1}\cdot p_{2}}{p_{3}}-\frac{\mu_{G}(I_{\mathbf{a}c})\cdot\mu_{G}(I_{\mathbf{ba}})}{\mu_{G}(I_{\mathbf{a}})}|<\frac{\delta}{2}\:.\label{E5}
\end{equation}

For each $i\ge1$ and $\mathbf{d}\in\cup_{k=1}^{\infty}\mathbb{N}^{k}$
denote by $E_{\mathbf{d},i}$ the event
\[
\left\{ A_{i}...A_{i+|\mathbf{d}|-1}=\mathbf{d}\right\} ,
\]
where $|\mathbf{d}|$ stands for the length of $\mathbf{d}$, and
set $p_{\mathbf{d},i}:=\mathbb{P}(E_{\mathbf{d},i})$. Let $\mathbf{d}\in\cup_{k=1}^{\infty}\mathbb{N}^{k}$
and assume
\[
\underset{n}{\limsup}\:\frac{1}{n}\#\left\{ 1\le i\le n\::\:p_{\mathbf{d},i}<\mu_{G}(I_{\mathbf{d}})-\epsilon\right\} >\frac{1}{10},
\]
then there exists $q\in\mathcal{Q}_{10}$ with
\begin{equation}
p_{\mathbf{d},q(i)}<\mu_{G}(I_{\mathbf{d}})-\epsilon\mbox{ for all }i\ge1\:.\label{E10}
\end{equation}
Since $\{A_{n}\}_{n=1}^{\infty}$ is $*$-mixing it is evident from
the definition that $\{1_{E_{\mathbf{d},q(i)}}\}_{i=1}^{\infty}$
is also $*$-mixing, where $1_{E}$ denotes the indicator of the event
$E$. By the law of large numbers for sums of $*$-mixing bounded
random variables (see Theorem 2 in \cite{BHK}),
\[
\underset{n}{\lim}\:\frac{1}{n}\sum_{i=1}^{n}\left(1_{E_{\mathbf{d},q(i)}}-p_{\mathbf{d},q(i)}\right)=0\:\mbox{ almost surely}\:.
\]
Hence for $\nu$-a.e. $x\in X$,
\[
\underset{n}{\lim}\:\left|\frac{1}{n}\sum_{i=1}^{n}\mathbb{I}_{\mathbf{d}}(T^{q(i)}x)-\frac{1}{n}\sum_{i=1}^{n}p_{\mathbf{d},q(i)}\right|=0\:.
\]
From this and (\ref{E10}) we get that for $\nu$-a.e. $x\in X$,
\[
\underset{n\rightarrow\infty}{\liminf}\:\left|\frac{1}{n}\sum_{i=1}^{n}\mathbb{I}_{\mathbf{d}}(T^{q(i)}x)-\mu_{G}(I_{\mathbf{d}})\right|=\underset{n\rightarrow\infty}{\liminf}\:\left|\frac{1}{n}\sum_{i=1}^{n}p_{\mathbf{d},q(i)}-\mu_{G}(I_{\mathbf{d}})\right|\ge\epsilon,
\]
which implies $\nu(\Gamma_{q,\mathbf{d}}^{\epsilon/2})=1$. Now by
Theorem \ref{T6}
\[
\dim_{H}(\nu)\le\dim_{H}(\Gamma_{q,\mathbf{d}}^{\epsilon/2})\le1-c_{10,\epsilon/2}\:.
\]
In a similar manner it can be shown that $\dim_{H}(\nu)\le1-c_{10,\epsilon/2}$
if
\[
\underset{n}{\limsup}\:\frac{1}{n}\#\left\{ 1\le i\le n\::\:p_{\mathbf{d},i}>\mu(I_{\mathbf{d}})+\epsilon\right\} >\frac{1}{10}\:.
\]

It follows that we can assume
\[
\underset{n}{\liminf}\:\frac{1}{n}\#\left\{ 1\le i\le n\::\:\begin{array}{c}
|p_{\mathbf{ba},i}-\mu(I_{\mathbf{ba}})|\le\epsilon,\\
|p_{\mathbf{a}c,i+m}-\mu(I_{\mathbf{a}c})|\le\epsilon,\\
|p_{\mathbf{a},i+m}-\mu(I_{\mathbf{a}})|\le\epsilon
\end{array}\right\} >\frac{1}{10},
\]
and so there exists $q\in\mathcal{Q}_{10}$ with
\begin{equation}
\left|p_{\mathbf{ba},q(i)}-\mu_{G}(I_{\mathbf{ba}})\right|,\:\left|p_{\mathbf{a}c,q(i)+m}-\mu_{G}(I_{\mathbf{a}c})\right|,\:\left|p_{\mathbf{a},q(i)+m}-\mu_{G}(I_{\mathbf{a}})\right|\le\epsilon\label{E6}
\end{equation}
for all $i\ge1$. Since $\{A_{n}\}_{n=1}^{\infty}$ is a Markov chain
of order $k$
\begin{equation}
p_{\mathbf{ba}c,q(i)}=\frac{p_{\mathbf{ba},q(i)}\cdot p_{\mathbf{a}c,q(i)+m}}{p_{\mathbf{a},q(i)+m}}\mbox{ for }i\ge1,\label{E7}
\end{equation}
where $p_{\mathbf{a},q(i)+m}>0$ by (\ref{E6}) and $\mu_{G}(I_{\mathbf{a}})>\epsilon$.
The sequence $\{1_{E_{\mathbf{ba}c,q(i)}}\}_{i=1}^{\infty}$ is $*$-mixing,
so by the law of large numbers for sums of $*$-mixing random variables,
\[
\underset{n}{\lim}\:\frac{1}{n}\sum_{i=1}^{n}\left(1_{E_{\mathbf{ba}c,q(i)}}-p_{\mathbf{ba}c,q(i)}\right)=0\:\mbox{ almost surely}\:.
\]
It follows that for $\nu$-a.e. $x\in X$,
\[
\underset{n}{\lim}\:\left|\frac{1}{n}\sum_{i=1}^{n}\mathbb{I}_{\mathbf{ba}c}(T^{q(i)}x)-\frac{1}{n}\sum_{i=1}^{n}p_{\mathbf{ba}c,q(i)}\right|=0\:.
\]
From this, (\ref{E7}), (\ref{E6}), (\ref{E5}) and (\ref{E8}) we
get that for $\nu$-a.e. $x\in X$,
\begin{align*}
\underset{n}{\liminf}\:\left|\frac{1}{n}\sum_{i=1}^{n}\mathbb{I}_{\mathbf{ba}c}(T^{q(i)}x)-\mu_{G}(I_{\mathbf{ba}c})\right|\\
=\underset{n}{\liminf}\:\left|\frac{1}{n}\sum_{i=1}^{n}p_{\mathbf{ba}c,q(i)}-\mu_{G}(I_{\mathbf{ba}c})\right|\\
\ge\left|\mu_{G}(I_{\mathbf{ba}c})-\frac{\mu_{G}(I_{\mathbf{a}c})\cdot\mu_{G}(I_{\mathbf{ba}})}{\mu_{G}(I_{\mathbf{a}})}\right|\\
-\underset{n}{\limsup}\:\frac{1}{n}\sum_{i=1}^{n}\left|\frac{p_{\mathbf{ba},q(i)}\cdot p_{\mathbf{a}c,q(i)+m}}{p_{\mathbf{a},q(i)+m}}-\frac{\mu_{G}(I_{\mathbf{a}c})\cdot\mu_{G}(I_{\mathbf{ba}})}{\mu_{G}(I_{\mathbf{a}})}\right| & \ge\delta/2\:.
\end{align*}
Hence $\nu(\Gamma_{q,\mathbf{ba}c}^{\delta/4})=1$, and so by Theorem
\ref{T6}
\[
\dim_{H}(\nu)\le\dim_{H}(\Gamma_{q,\mathbf{ba}c}^{\delta/4})\le1-c_{10,\delta/4}\:.
\]
This completes the proof of the theorem.
\end{proof}

\section{\label{S5}Construction of the measures $\nu_{K}$}

In the proof below we use the notation for the Kolmogorov-Sinai entropy
from Chapter 4 of \cite{W2}. In particular the entropy of a Borel
probability measure $\theta$ on $X$, with respect to a countable
Borel partition $\xi$ of $X$, is denoted by $H_{\theta}(\xi)$.
If $\mathcal{F}$ is a sub-$\sigma$-algebra of the Borel $\sigma$-algebra
of $X$, then $H_{\theta}(\xi\mid\mathcal{F})$ is the entropy of
$\theta$ with respect to $\xi$ conditioned on $\mathcal{F}$. If
$\theta$ is $T$-invariant the entropy of $T$ with respect to $\theta$
is denoted by $h_{\theta}$. If $\theta$ is also ergodic we write
$\gamma_{\theta}$ for the Lyapunov exponent of the system $(X,T,\theta)$,
i.e.
\[
\gamma_{\theta}=\int_{X}\log\left|T'\left(x\right)\right|\:d\theta(x)=-2\int_{X}\log x\:d\theta(x)\:.
\]
Given $a_{1},...,a_{m}\in\mathbb{N}$ we denote by $[a_{1},...,a_{m}]$
the finite continued fraction which lies in $(0,1)$ and has coefficients
$a_{1},...,a_{m}$, i.e.
\[
[a_{1},...,a_{m}]=\frac{1}{a_{1}+\frac{1}{a_{2}+\cdots\frac{1}{a_{m-1}+\frac{1}{a_{m}}}}}\:.
\]

In order to establish the $\psi$-mixing property in the proof of
Claim \ref{C5} we shall need the following proposition. It follows
directly from Theorem 1 in \cite{Br2}.
\begin{prop}
\label{P5.1}Let $\{A_{n}\}_{n=1}^{\infty}$ be a stationary and mixing
sequence of random variables. Assume there exists a constant $0<C<\infty$
with
\[
C^{-1}\le\frac{\mathbb{P}(E\cap F)}{\mathbb{P}(E)\mathbb{P}(F)}\le C
\]
for all $l\ge1$, $E\in\sigma\{A_{1},...,A_{l}\}$ and $F\in\sigma\{A_{l+1},A_{l+2},...\}$.
Then $\{A_{n}\}_{n=1}^{\infty}$ is $\psi$-mixing.
\end{prop}

\begin{proof}[Proof of Claim \ref{C5}]
Fix $k\ge3$ and for every $\mathbf{a}\in\mathbb{N}^{k}$ and $c\in\mathbb{N}$
set
\[
p_{\mathbf{a}}=\mu_{G}(I_{\mathbf{a}})\mbox{ and }p_{\mathbf{a},c}=\frac{\mu_{G}(I_{\mathbf{a}c})}{\mu_{G}(I_{\mathbf{a}})}\:.
\]
Then $\sum_{c\in\mathbb{N}}p_{\mathbf{a},c}=1$ for each $\mathbf{a}\in\mathbb{N}^{k}$
and $p=\{p_{\mathbf{a}}\}_{\mathbf{a}\in\mathbb{N}^{k}}$ is a probability
vector. Let $\{A_{n}\}_{n=1}^{\infty}$ be the $k$-step $\mathbb{N}$-valued
Markov chain corresponding to the transition probabilities $\{p_{\mathbf{a},c}\}_{(\mathbf{a},c)\in\mathbb{N}^{k+1}}$
and initial distribution $\{p_{\mathbf{a}}\}_{\mathbf{a}\in\mathbb{N}^{k}}$.
For each $\mathbf{b}\in\mathbb{N}^{k-1}$ and $d\in\mathbb{N}$
\[
\sum_{c\in\mathbb{N}}p_{c\mathbf{b}}\cdot p_{c\mathbf{b},d}=\sum_{c\in\mathbb{N}}\mu_{G}(I_{c\mathbf{b}})\cdot\frac{\mu_{G}(I_{c\mathbf{b}d})}{\mu_{G}(I_{c\mathbf{b}})}=\mu_{G}(T^{-1}(I_{\mathbf{b}d}))=p_{\mathbf{b}d},
\]
hence $\{A_{n}\}_{n=1}^{\infty}$ is stationary. Considering $\{A_{n}\}_{n=1}^{\infty}$
as a $1$-step Markov chain on the state space $\mathbb{N}^{k}$,
it is easy to see it is irreducible and aperiodic. From this and Theorem
8.6 in \cite{B} it follows $\{A_{n}\}_{n=1}^{\infty}$ is mixing.

Let us show $\{A_{n}\}_{n=1}^{\infty}$ is in fact $\psi$-mixing.
From (3.22) in chapter 3 of \cite{EW} it follows there exists a constant
$1<C<\infty$ with,
\begin{equation}
C^{-l}\le\frac{\mu_{G}(I_{(a_{1},...,a_{l})})}{\mu_{G}(I_{a_{1}})\cdot...\cdot\mu_{G}(I_{a_{l}})}\le C^{l}\:\mbox{ for }l\ge1\mbox{ and }a_{1},...,a_{l}\in\mathbb{N}\:.\label{E15}
\end{equation}
For $l,m>k$ , $(a_{1},...,a_{l})=\mathbf{a}\in\mathbb{N}^{l}$ and
$(b_{1},...,b_{m})=\mathbf{b}\in\mathbb{N}^{m}$ set
\[
R:=\frac{\mathbb{P}\{A_{1}...A_{l+m}=\mathbf{ab}\}}{\mathbb{P}\{A_{1}...A_{l}=\mathbf{a}\}\mathbb{P}\{A_{1}...A_{m}=\mathbf{b}\}},
\]
then
\[
R=\frac{1}{\mu_{G}(I_{(b_{1},...,b_{k})})}\cdot\prod_{j=1}^{k}\frac{\mu_{G}(I_{(a_{l-k+j},...,a_{l},b_{1},...,b_{j})})}{\mu_{G}(I_{(a_{l-k+j},...,a_{l},b_{1},...,b_{j-1})})}\:.
\]
This together with (\ref{E15}) gives
\[
C^{-2k(k+1)}\le R\le C^{2k(k+1)}\:.
\]
From Proposition \ref{P5.1}, combined with a monotone class argument,
it now follows that $\{A_{n}\}_{n=1}^{\infty}$ is $\psi$-mixing.

Let $\nu$ be the distribution of $[A_{1},A_{2},...]$, then $\nu$
is $T$-invariant and ergodic. In order to prove the claim it remains
to show that $\dim_{H}\nu\ge1-2^{3-k}$. Set
\[
\xi=\{I_{c}\::\:c\in\mathbb{N}\},
\]
then it is easy to check that
\[
H_{\nu}(\xi)=H_{\mu_{G}}(\xi)<\infty
\]
and
\[
\sum_{c\in\mathbb{N}}\nu(I_{c})\log c=\sum_{c\in\mathbb{N}}\mu_{G}(I_{c})\log c<\infty,
\]
which shows $h_{\nu},\gamma_{\nu},h_{\mu_{G}}\mbox{ and }\gamma_{\mu_{G}}$
are all finite. From this and Section 2 of \cite{BH} it follows that
\begin{equation}
\dim_{H}\nu=\frac{h_{\nu}}{\gamma_{\nu}}\mbox{ and }1=\dim_{H}\mu_{G}=\frac{h_{\mu_{G}}}{\gamma_{\mu_{G}}}\:.\label{E11}
\end{equation}
Moreover, it is well known
\begin{equation}
\gamma_{\mu_{G}}=-\frac{2}{\log2}\int\frac{\log x}{1+x}\:dx=\frac{\pi^{2}}{6\log2}>2\:.\label{E14}
\end{equation}

By an argument similar to the one given in Theorem 4.27 in \cite{W2},
\[
h_{\nu}=-\sum_{\mathbf{a}\in\mathbb{N}^{k}}\sum_{c\in\mathbb{N}}p_{\mathbf{a}}p_{\mathbf{a},c}\log p_{\mathbf{a},c}\:.
\]
From this and the definition of conditional entropy,
\[
h_{\nu}=H_{\mu_{G}}(\vee_{j=0}^{k}T^{-j}\xi\mid\vee_{j=0}^{k-1}T^{-j}\sigma(\xi))\:.
\]
Now from Theorems 4.3 and 4.14 in \cite{W2},
\begin{multline}
h_{\nu}=H_{\mu_{G}}(\vee_{j=0}^{k}T^{-j}\xi)-H_{\mu_{G}}(\vee_{j=1}^{k}T^{-j}\xi)\\
=H_{\mu_{G}}(\xi\mid\vee_{j=1}^{k}T^{-j}\sigma(\xi))\ge H_{\mu_{G}}(\xi\mid\vee_{j=1}^{\infty}T^{-j}\sigma(\xi))=h_{\mu_{G}}\:.\label{E12}
\end{multline}

Assume $k$ is even for the moment, then
\[
[a_{1},...,a_{k}]\le x\le[a_{1},...,a_{k}+1]
\]
for every $(a_{1},...,a_{k})=\mathbf{a}=\mathbb{N}^{k}$ and $x\in I_{\mathbf{a}}$.
It follows that,
\begin{multline*}
\gamma_{\nu}-\gamma_{\mu_{G}}=-2\int_{X}\log x\:d\nu(x)+2\int_{X}\log x\:d\mu_{G}(x)\\
=2\sum_{\mathbf{a}\in\mathbb{N}^{k}}\left(\int_{I_{\mathbf{a}}}\log\frac{1}{x}\:d\nu(x)+\int_{I_{\mathbf{a}}}\log x\:d\mu_{G}(x)\right)\\
\le2\sum_{a_{1},...,a_{k}\in\mathbb{N}}\left(\int_{I_{(a_{1},...,a_{k})}}\log\frac{1}{[a_{1},...,a_{k}]}\:d\nu(x)+\int_{I_{(a_{1},...,a_{k})}}\log[a_{1},...,a_{k}+1]\:d\mu_{G}(x)\right)\\
=2\sum_{a_{1},...,a_{k}\in\mathbb{N}}\mu_{G}(I_{(a_{1},...,a_{k})})\cdot\log\frac{[a_{1},...,a_{k}+1]}{[a_{1},...,a_{k}]}\:.
\end{multline*}
Fix $a_{1},...,a_{k}\in\mathbb{N}$, then
\[
\log\frac{[a_{1},...,a_{k}+1]}{[a_{1},...,a_{k}]}\le\frac{[a_{1},...,a_{k}+1]-[a_{1},...,a_{k}]}{[a_{1},...,a_{k}]}\:.
\]
Let $p,q\in\mathbb{N}$ be with $\mathrm{gcd}(p,g)=1$ and $\frac{p}{q}=[a_{1},...,a_{k}]$.
From inequalities (3.6), (3.7) and (3.14) in \cite{EW} it follows
that $q,p\ge2^{(k-2)/2}$ and
\[
[a_{1},...,a_{k}+1]-[a_{1},...,a_{k}]\le q^{-2}\:.
\]
Hence
\[
\log\frac{[a_{1},...,a_{k}+1]}{[a_{1},...,a_{k}]}\le\frac{1/q^{2}}{p/q}=\frac{1}{pq}\le2^{2-k},
\]
and so $\gamma_{\nu}-\gamma_{\mu_{G}}\le2^{3-k}$. By exchanging between
$\gamma_{\mu_{G}}$ and $\gamma_{\nu}$ it can be shown that $\gamma_{\mu_{G}}-\gamma_{\nu}\le2^{3-k}$.
From $k\ge3$ and (\ref{E14}) we get $\gamma_{\nu}\ge1$, hence
\begin{equation}
1\le\frac{\gamma_{\mu_{G}}}{\gamma_{\nu}}+\frac{2^{3-k}}{\gamma_{\nu}}\le\frac{\gamma_{\mu_{G}}}{\gamma_{\nu}}+2^{3-k}\:.\label{E13}
\end{equation}
A similar argument shows (\ref{E13}) holds when $k$ is odd. From
(\ref{E11}), (\ref{E12}) and (\ref{E13}) we now get
\[
\dim_{H}\nu=\frac{\gamma_{\mu_{G}}}{\gamma_{\nu}}\cdot\frac{h_{\nu}}{\gamma_{\mu_{G}}}\ge(1-2^{3-k})\cdot\frac{h_{\mu_{G}}}{\gamma_{\mu_{G}}}=1-2^{3-k},
\]
which completes the proof of the claim.
\end{proof}

\section{\label{S6}Extension of results for $f$-expansions}

With almost no change, Theorems \ref{T3} and \ref{T6} extend to
the more general setup of $f$-expansions, which we now define. Let
$M\in\{2,3,...\}\cup\{\infty\}$. Let $f$ be either a strictly decreasing
continuous function defined on $[1,M+1]$ with $f(1)=1$ and $f(M+1)=0$,
or a strictly increasing continuous function defined on $[0,M]$ with
$f(0)=0$ and $f(M)=1$. For $x\in(0,1)$ set $r_{0}(x)=x$ and $r_{i+1}(x)=\left\{ f^{-1}(r_{i}(x))\right\} $
for $i\ge0$, where $\left\{ \cdot\right\} $ denotes the fractional
part of a number. Let $X$ be the set of all $x\in(0,1)$ with $r_{i}(x)\ne0$
for every $i\ge0$, then $(0,1)\setminus X$ is clearly countable.
Write
\[
\mathcal{N}=\{[y]\::\:y\in f^{-1}(0,1)\},
\]
where $\left[\cdot\right]$ is the integer part of a number. For $x\in X$
and $i\ge1$ set
\[
\alpha_{i}(x)=\left[f^{-1}(r_{i-1}(x))\right],
\]
then $\alpha_{i}(x)\in\mathcal{N}$. We shall assume that
\begin{equation}
x=f(\alpha_{1}(x)+f(\alpha_{2}(x)+f(\alpha_{3}(x)+...)))\mbox{ for all }x\in X,\label{E90}
\end{equation}
and call the expression on the right hand side the $f$-expansion
of $x$. Regularity conditions on $f$ were given by Rényi \cite{R},
which ensure that (\ref{E90}) is satisfied. The main example of the
decreasing case is $f(x)=1/x$, which leads to the continued fraction
expansion, and of the increasing case is $f(x)=x/M,$ which leads
to the base-$M$ expansion. For more details on $f$-expansions see
\cite{R}, \cite{KP}, \cite{H} and the references therein.

We use the notation $I_{\mathbf{a}}$ and $\mathbb{I}_{\mathbf{a}}$,
introduced in Section \ref{S2}, with $X$ and $\alpha_{i}$ as defined
in this section and $\mathbf{a}\in\cup_{k=1}^{\infty}\mathcal{N}^{k}$.
For $x\in(0,1)$ set $Tx=f^{-1}x-\left[f^{-1}x\right]$, then $\alpha_{i}(Tx)=\alpha_{i+1}(x)$
for $x\in X$. We shall assume that
\begin{enumerate}
\item \label{C1}the restriction of $T$ to $f(a,a+1)$ is $C^{2}$ for
each $a\in\mathcal{N}$;
\item there exists $\ell\in\mathbb{N}$ and $\beta>0$ with $|(T^{\ell})'(x)|\ge\beta$
for all $x\in X$;
\item \label{C3}there exists $1\le Q<\infty$ with $\left|\frac{T''(x)}{T'(y)T'(z)}\right|\le Q$
for all $a\in\mathcal{N}$ and $x,y,z\in I_{a}$.
\end{enumerate}
Then by Theorem 22 in \cite{W}, there exists an absolutely continuous
$T$-invariant mixing probability measure $\mu_{T}$ on $X$, such
that $0<\frac{d\mu_{T}}{d\mathcal{L}}\in C[0,1]$. Here, as above,
$\mathcal{L}$ is the Lebesgue measure.

For $q\in\mathcal{Q}_{L}$ with $\mathcal{Q}_{L}$ defined in Section
2, $\mathbf{a}\in\cup_{k=1}^{\infty}\mathcal{N}^{k}$ and $\delta>0$
let
\[
\Gamma_{q,\mathbf{a}}^{\delta}=\{x\in X\::\:\underset{n\rightarrow\infty}{\liminf}\:\left|\frac{1}{n}\sum_{i=1}^{n}\mathbb{I}_{\mathbf{a}}(T^{q(i)}x)-\mu_{T}(I_{\mathbf{a}})\right|>\delta\}\:.
\]
The following theorem is an analogue of Theorem \ref{T6}, and can
be proven in exactly the same manner.
\begin{thm}
\label{T6.1}Suppose that $T$ satisfies conditions (\ref{C1})-(\ref{C3})
and assume, in addition, that for some $t<1$,
\begin{equation}
\underset{x\in X}{\sup}\:\sum_{y:Ty=x}|T'(y)|^{-t}<\infty\:.\label{E91}
\end{equation}
Then for every $L>1$ and $\delta>0$ there exists $0<c_{f,L,\delta}<1$
with
\[
\sup\{\dim_{H}(\Gamma_{q,\mathbf{a}}^{\delta})\::\:q\in\mathcal{Q}_{L},\:\mathbf{a}\in\cup_{k=1}^{\infty}\mathcal{N}^{k}\}\le1-c_{f,L,\delta}\:.
\]

\end{thm}

\begin{rem}
The condition (\ref{E91}) is needed in order to apply Theorem 4.1
from \cite{KPW}, as we did at the beginning of the proof of Theorem
\ref{T6}. Since $\{\alpha_{i}\}_{i=1}^{\infty}$ is a $\psi$-mixing
sequence with respect to $\mu_{T}$ (see \cite{A} or \cite{H}),
the large deviations estimate from Corollary \ref{P8} is valid for
$\mu_{T}$. Now the proof of Theorem \ref{T6.1} follows almost verbatim
the proof of Theorem \ref{T6}.
\end{rem}

An important ingredient in the proof of Theorem \ref{T3} is the fact
that, for any $k\ge0$, the continued fraction digits under $\mu_{G}$
do not form a $k$-step Markov chain. Hence, in order to generalize
Theorem \ref{T3} to the case of $f$-expansions we shall need the
following lemma. For $t\in[0,1]$ set $F(t)=\mu_{T}\left([0,t]\right)$,
and let $S=F\circ T\circ F^{-1}$. Since $F'=\frac{d\mu_{T}}{d\mathcal{L}}\in C[0,1]$
with $\frac{d\mu_{T}}{d\mathcal{L}}>0$, $F$ is a diffeomorphism
of $[0,1]$ onto itself. Given $a\in\mathcal{N}$ write $\widetilde{I_{a}}:=f(a,a+1)$.
\begin{lem}
\label{L6.3}Assume the digits $\{\alpha_{i}\}_{i=1}^{\infty}$ of
the $f$-expansion are not independent under $\mu_{T}$. Then $\{\alpha_{i}\}_{i=1}^{\infty}$
do not form a $k$-step Markov chain under $\mu_{T}$ for any $k\ge1$.
\end{lem}

\begin{proof}
Note that $F\mu_{T}=\mathcal{L}$ and $S\mathcal{L}=\mathcal{L}$.
From the chain rule it follows that for every $a\in\mathcal{N}$ and
$x\in F\widetilde{I_{a}}$,
\[
S'(x)=F'(TF^{-1}x)T'(F^{-1}x)\left(F'(F^{-1}x)\right)^{-1},
\]
and so $S'$ is continuous on $F\widetilde{I_{a}}$. Let $\beta_{1}:FX\rightarrow\mathcal{N}$
be such that $\beta_{1}(x)=a$ for $a\in\mathcal{N}$ and $x\in FI_{a}$.
For $i\ge1$ set $\beta_{i}=\beta_{1}\circ S^{i-1}$, then $\beta_{i}=\alpha_{i}\circ F^{-1}$.
Given $(a_{1},...,a_{l})=\mathbf{a}\in\mathcal{N}^{l}$ let
\[
J_{\mathbf{a}}=\{x\in FX\::\:\beta_{i}(x)=a_{i}\mbox{ for }1\le i\le l\},
\]
then $J_{\mathbf{a}}=FI_{\mathbf{a}}$. Note that 
\begin{equation}
\mathcal{L}(J_{\mathbf{a}})=\mu_{T}(I_{\mathbf{a}})\mbox{ for every }l\ge1\mbox{ and }\mathbf{a}\in\mathcal{N}^{l}\:.\label{E99}
\end{equation}

Let $k\ge1$ and assume by contradiction that $\{\alpha_{i}\}_{i=1}^{\infty}$
forms a $k$-step Markov chain under $\mu_{T}$. From this, since
$\{\alpha_{i}\}_{i=1}^{\infty}$ are not independent, and from (\ref{E99}),
it follows that under $\mathcal{L}$ the variables $\{\beta_{i}\}_{i=1}^{\infty}$
form a stationary $k$-step Markov chain but are not independent.
Since $\{\beta_{i}\}_{i=1}^{\infty}$ is a stationary $k$-step Markov
chain,
\[
\mathcal{L}\{\beta_{1}=c\mid S^{-1}(J_{\mathbf{b}})\}=\mathcal{L}\{\beta_{1}=c\mid S^{-1}(J_{\mathbf{bf}})\}
\]
for every $c\in\mathcal{N}$, $\mathbf{b}\in\mathcal{N}^{k}$, $l\ge1$
and $\mathbf{f}\in\mathcal{N}^{l}$. It follows there exist $c\in\mathcal{N}$
and $\mathbf{b},\mathbf{d}\in\mathcal{N}^{k}$ with,
\begin{equation}
\mathcal{L}\{\beta_{1}=c\mid S^{-1}(J_{\mathbf{b}})\}\ne\mathcal{L}\{\beta_{1}=c\mid S^{-1}(J_{\mathbf{d}})\},\label{E97}
\end{equation}
otherwise it would hold that $\{\beta_{i}\}_{i=1}^{\infty}$ are independent
under $\mathcal{L}$.

It is not hard to see that for $\mathcal{L}$-a.e. $x\in J_{c}$,
\begin{equation}
\mathcal{L}\{\beta_{1}=c\mid\sigma\{\beta_{2},\beta_{3},...\}\}(x)=\left(S'(x)\right)^{-1},\label{E98}
\end{equation}
where the left hand side is the conditional $\mathcal{L}$-probability
of the event $\{\beta_{1}=c\}$ with respect to the $\sigma$-algebra
$\sigma\{\beta_{2},\beta_{3},...\}$. Let $\mathbf{a}\in\mathcal{N}^{k}$.
Then since $\{\beta_{i}\}_{i=1}^{\infty}$ is a $k$-step Markov chain
under $\mathcal{L}$, it follows for $\mathcal{L}$-a.e. $x\in J_{c\mathbf{a}}$
that
\begin{multline*}
\mathcal{L}\{\beta_{1}=c\mid\sigma\{\beta_{2},\beta_{3},...\}\}(x)\\
=\mathcal{L}\{\beta_{1}=c\mid\sigma\{\beta_{2},...,\beta_{k+1}\}\}(x)=\mathcal{L}\{\beta_{1}=c\mid S^{-1}(J_{\mathbf{a}})\}\:.
\end{multline*}
This together with (\ref{E98}) shows that
\begin{equation}
\left(S'(x)\right)^{-1}=\mathcal{L}\{\beta_{1}=c\mid S^{-1}(J_{\mathbf{a}})\}\mbox{ for }\mathcal{L}\mbox{-a.e. }x\in J_{c\mathbf{a}}\:.\label{E89}
\end{equation}
Since $S'$ is continuous on $F\widetilde{I_{c}}$ and
\[
F(\widetilde{I_{c}}\cap X)=\cup_{\mathbf{a}\in\mathcal{N}^{k}}J_{c\mathbf{a}},
\]
it follows easily from (\ref{E89}) that $S'$ must be constant on
$F\widetilde{I_{c}}$. On the other hand, by (\ref{E97}) and (\ref{E89})
this is not possible. We have thus reached a contradiction, which
shows that $\{\alpha_{i}\}_{i=1}^{\infty}$ does not form a $k$-step
Markov chain under $\mu_{T}$.
\end{proof}

\begin{rem}
In Proposition 7.1 from \cite{KPW} it is shown that $\{\alpha_{i}\}_{i=1}^{\infty}$
are independent under $\mu_{T}$ if and only if $S$ is linear on
$F\widetilde{I_{a}}$ for each $a\in\mathcal{N}$. From this and Lemma
\ref{L6.3} it follows that if $S$ is not linear on $F\widetilde{I_{a}}$
for some $a\in\mathcal{N}$, then $\{\alpha_{i}\}_{i=1}^{\infty}$
do not form a $k$-step Markov chain under $\mu_{T}$ for any $k\ge0$.
\end{rem}

The following theorem is an analogue, for the case of $f$-expansions,
of Theorem \ref{T3} above and Corollary 2.3 from \cite{KPW}. It
can be derived from Theorem \ref{T6.1}, Theorem 2.1 in \cite{KPW},
and Lemma \ref{L6.3}, by an argument similar to the one given in
the proof of Theorem \ref{T3}. Given $a_{1},a_{2},...\in\mathcal{N}$
denote by $[a_{1},a_{2},...]$ the unique $x\in X$ with $\alpha_{i}(x)=a_{i}$
for $i\ge1$.
\begin{thm}
Suppose that $T$ satisfies the conditions (\ref{C1})-(\ref{C3})
and, in addition, that (\ref{E91}) holds for some $t<1$. Assume
the digits $\{\alpha_{i}\}_{i=1}^{\infty}$ of the $f$-expansion
are not independent under $\mu_{T}$. Let $k\ge0$ and let $\{A_{n}\}_{n=1}^{\infty}$
be an $\mathcal{N}$-valued $k$-step Markov chain (when $k=0$ this
means $A_{1},A_{2},...$ are independent). Assume $\{A_{n}\}_{n=1}^{\infty}$
is $*$-mixing or that it is stationary and ergodic. Let $\nu$ be
the distribution of the random variable $[A_{1},A_{2},...]$. Then
$\dim_{H}(\nu)\le1-c_{f,k}$, where $0<c_{f,k}<1$ is a constant depending
only on $f$ and $k$.
\end{thm}

\end{document}